\documentclass[11pt]{article}

    \usepackage{amsmath, amssymb,amsfonts,xspace,graphicx,relsize,bm,mathtools,xcolor}
    \usepackage{mathrsfs,hyperref}
    \usepackage{enumerate}
    \usepackage{bm}

    \usepackage{amsthm,cleveref}
    \usepackage{lipsum} 
    \newtheorem{theorem}{Theorem}

    \def\01{\{0,1\}}

    \newcommand{\Tr}{\mbox{\rm Tr}}

    \newcommand{\norm}[1]{\mbox{$\|{#1}\|$}}
    
    \newcommand{\bx}{\boldsymbol{x}}
    \newcommand{\bA}{\boldsymbol{A}}

    \newcommand{\bg}{\boldsymbol{g}}
    \newcommand{\boldr}{\boldsymbol{r}}
    \newcommand{\by}{\boldsymbol{y}}
    \newcommand{\bz}{\boldsymbol{z}}

    \newcommand{\bb}{\boldsymbol{b}}
    \newcommand{\bW}{\boldsymbol{W}}
    \newcommand{\bn}{\boldsymbol{b}}
    \newcommand{\bv}{\boldsymbol{v}}

    \def\H{\mathcal{H}}
    
    \newcommand{\G}{\mathcal{G}}
    \newcommand{\Sh}{\mathcal{T}}
     \newcommand{\V}{{\mathcal{V}}}
    
     \newcommand{\De}{{\mathcal{D}}}
     \newcommand{\R}{{\mathbb{R}}}
     \newcommand{\id}{{\mathbb{I}}}

    \newtheorem{fact}[theorem]{Fact}
    \newtheorem{lemma}[theorem]{Lemma}

    \newtheorem{claim}[theorem]{Claim}

    \def\01{\{0,1\}}

    \newcommand{\GSA}{\mathsf{GSA}}

    \newcommand{\reals}{{\mathbb R}}
    \newcommand {\br} [1] {{\left(#1\right)}}
     \newcommand {\set} [1] {{ \left\lbrace #1 \right\rbrace }}
    \newcommand {\Br} [1] {{ \big[ #1 \big] }}

    \usepackage[margin=1in]{geometry}
    \hypersetup{
    	colorlinks,
    	linkcolor={red!100!black},
    	citecolor={blue!100!black},
    }

\begin{document}

\title{On the Gaussian surface area of spectrahedra}

\author{Srinivasan Arunachalam\thanks{IBM T.J. Watson Research Center \textsf{Srinivasan.Arunachalam@ibm.com}} \and
Oded Regev\thanks{Courant Institute of Mathematical Sciences, New York University, \textsf{regev@cims.nyu.edu}}
 \quad \and
Penghui Yao\thanks{State Key Laboratory for Novel Software Technology, Nanjing
    University, \textsf{pyao@nju.edu.cn}}\vspace{1mm}
 \and
}

\maketitle

\begin{abstract}
We show that for sufficiently large $n\geq 1$ and $d=C n^{3/4}$ for some universal constant $C>0$, a random spectrahedron with matrices drawn from Gaussian orthogonal ensemble has Gaussian surface area $\Theta(n^{1/8})$ with high~probability.
\end{abstract}

\section{Introduction}
A \emph{spectrahedron} $S\subseteq\reals^n$ is a set of the form
$$
S=\set{x\in\reals^n:\sum_ix_iA^{(i)}\preceq B} \; ,
$$
for some $d\times d$ symmetric matrices
 $A^{(1)},\ldots, A^{(n)}, B\in \mathrm{Sym}_d$.
Here we will be concerned with the \emph{Gaussian surface area} of $S$, defined as
\begin{align}\label{eq:gsaoriginaldefinitionwithouter}
\GSA\br{S}=\liminf_{\delta\rightarrow 0}\frac{\G^n\br{S_{\delta}^{\mathrm{out}}}}{\delta} \; ,
\end{align}
where $S_{\delta}^{\mathrm{out}}=\set{x \notin S:\mathrm{dist}(x,S)\leq\delta}$ denotes the outer $\delta$-neighborhood of $S$ under Euclidean distance  
and $\G^n(\cdot)$ denotes the standard Gaussian measure on $\mathbb{R}^n$ whose density is $(2\pi)^{-n/2} \exp(-\|x\|^2/2)$.
Ball showed that the $\GSA$ of any convex body in $\mathbb{R}^n$ is $O(n^{1/4})$~\cite{ball1993reverse}, which was later shown to be tight by Nazarov~\cite{nazarov2003maximal}. Moreover,  Nazarov~\cite{klivans2008learning} showed that the $\GSA$ of a $d$-facet polytope\footnote{A $d$-facet polytope is the special case of a spectrahedron when the matrices, 
$A^{(1)},\ldots,A^{(n)},B$ are \emph{diagonal}.} in $\reals^n$ is $O(\sqrt{\log d})$ and this fact has found application in learning theory and constructing pseudorandom generators for polytopes~\cite{klivans2008learning,harsha2013invariance,servedio2017fooling,chattopadhyay2019simple}. We refer the interested reader to~\cite{klivans2008learning,harsha2013invariance} for more details.  Motivated by recent work~\cite{arunachalam2021positive}, this raises the question of whether the $\GSA$ of spectrahedra is also small. In this note we 
answer this question in the negative.
Recall that a matrix $\bA$ drawn from the Gaussian orthogonal ensemble is a symmetric matrix whose entries $\set{\bA_{i,j}}_{i\leq j}$ are all independent normal random variables of mean $0$ having variance $1$ if $i<j$ and variance $2$ if $i=j$. 

\begin{theorem}
\label{thm:gsa}
For a universal constant $C>0$ and any integers $n,d \ge 1$ satisfying $d\leq n/C$ the following hold. 
If  $\bA^{(1)},\ldots,\bA^{(n)}$ are i.i.d.~drawn from the $d\times d$ Gaussian orthogonal ensemble,  then the  spectrahedron
\begin{equation}\label{eq:main theorem spectrahedron}
\Sh=\Big\{x\in \R^n:\sum_i x_i \bA^{(i)}\preceq 2 \sqrt{nd}\cdot\id \Big\}
\end{equation}
satisfies $\GSA(\Sh)\geq c\cdot\sqrt{n/d}$ for some absolute constant $c>0$ with probability at least $1-C \exp(-d n^{-3/4} / C)$. 
Moreover, for any integer $d$ satisfying $d\leq n/C$,
$\GSA(\Sh)\leq2\sqrt{n}/(\sqrt{\pi d})$
holds with probability at least $1 - \exp(-n/50)$.
\end{theorem}


The theorem shows the existence of spectrahedra with $\GSA$ of $\Omega(n^{1/8})$. (In fact, a random spectrahedron as above satisfies this with constant probability). This lower bound can be contrasted with the $\GSA$ upper bound of Ball~\cite{ball1993reverse} of $O(n^{1/4})$ for \emph{arbitrary} convex bodies.
Moreover, the lower bound shows that in contrast to the case of polytopes, the $\GSA$ of spectrahedra  can depend polynomially on $d$.
A natural open question is how large the $\GSA$ of arbitrary spectrahedra can be; can spectrahedra with small $d$ (say, polynomial in $n$) achieve a $\GSA$ of $\Theta(n^{1/4})$?

\section{Preliminaries}\label{sec:pre}

For a matrix $A$, $\lambda_{\max}(A)$ is the maximum eigenvalue of $A$. We use $\bg,\bx,\bA$ to denote random variables.
We let $\G(0,\sigma^2)$ be the normal distribution with mean $0$ and variance $\sigma^2$. We denote by $\H_{d}$ the $d\times d$ Gaussian orthogonal ensemble (GOE).  Namely, $\bA\sim\H_{d}$ if
it is a symmetric matrix with entries $\set{\bA_{i,j}}_{i\leq j}$ independently distributed satisfying $\bA_{i,j}\sim \G(0,1)$ for $i<j$ and $\bA_{i,i}\sim \G(0,2)$. 
To keep notations short, for $b\geq 0$ we use $[a\pm b]$ to represent the interval $[a-b, a+b]$. For every $c\geq 0$, we use $c\cdot[a\pm b]$ to represent the interval $[ac\pm bc]$. We denote the set of $n$-dimensional unit vectors by $S^{n-1}$. Finally, we let $\chi_n$ be the $\chi$ distribution with $n$ degrees of freedom, which is the square root of the sum of the squares of $n$ independent standard normal variables. The following are some simple facts about the $\chi$ distribution.

\begin{fact}
\label{fact:chidistribution}
Let $n\in \mathbb{Z}_{>0}$ and $h(\cdot)$ be the pdf of $\chi_n$. Then the following hold.
\begin{enumerate}
    \item $h(x)\geq c$ for $x\in[\sqrt{n}\pm c]$, where $c>0$ is an absolute constant. 
    
    \item  $h(x)\leq 
    \sqrt{n}/(\sqrt{\pi}\cdot |x|)$ for $x\in\reals$.
\end{enumerate}
\end{fact}

\begin{proof}
Recall that by definition
$$
h(x)=\frac{1}{2^{\frac{n}{2}-1}\Gamma(\frac{n}{2})}x^{n-1}e^{-x^2/2}
$$
for $x\geq 0$, where $\Gamma(\cdot)$ denotes the gamma function, and $h(x)=0$ otherwise. 
By elementary calculus, $x^{n-1}e^{-x^2/2}$ monotonically increases  for  $0\leq x< \sqrt{n-1}$ and monotonically decreases for $x>\sqrt{n-1}$. We therefore have
\begin{equation}\label{eqn:2}
    x^{n-1}e^{-x^2/2}
    \geq
    \min\set{(\sqrt{n}+c)^{n-1}e^{-(\sqrt{n}+c)^2/2},(\sqrt{n}-c)^{n-1}e^{-(\sqrt{n}-c)^2/2}}
\end{equation}
 for $0<c\leq 1$ and $x\in[\sqrt{n}\pm c]$.
Item 1 now follows from Eq.~\eqref{eqn:2} and the fact that $\Gamma(z)\leq\sqrt{2\pi}z^{z-1/2}e^{-z+1/(12z)}$ for all $z>0$~\cite{stirling,jameson_2015}.

Item 2 is trivial for $x\leq 0$.  For $x>0$, it follows from the inequalities $\Gamma(z)\geq \sqrt{2\pi}z^{z-1/2}e^{-z}$ for all $z>0$~\cite{stirling,jameson_2015} and  $x^n e^{-x^2/2}\leq n^{n/2}e^{-n/2}$, which follows from the same argument as~above.
\end{proof}

\begin{lemma}[{\cite[comment below Lemma 1]{10.1214/aos/1015957395}}]\label{lem:chi2}
  For $n\geq 1$, let   $\boldr$ be a random variable distributed according to $\chi_n$. Then for every $x>0$, we have
  \[
  \Pr\Br{n-2\sqrt{nx}\leq \boldr^2\leq n+2\sqrt{nx}+2x} \geq 1-2e^{-x} \; .
  \]
\end{lemma}

For our purposes, it will be convenient to use an alternative definition of Gaussian surface area in terms of the \emph{inner} surface area. Namely, for
$S_{\delta}^{\mathrm{in}}=\set{x \in S:\mathrm{dist}(x,S^c}\leq\delta)$ where $S^c$ is the complement of the body $S$, we define,
\begin{align}
    \label{eq:innerGSA}
    \GSA\br{S}=\lim_{\delta\rightarrow 0}\frac{\G^n\br{S^{\text{in}}_{\delta}}}{\delta} \; .
\end{align}
It follows from Huang et al.~\cite[Theorem 3.3]{huang2021minkowski} that this definition is equivalent to the one in Eq.~\eqref{eq:gsaoriginaldefinitionwithouter} when $S$ is a convex body that contains the origin, which is sufficient for our purposes. 

To prove our main theorem, we use the following facts, starting with a well known bound on the size of an $\varepsilon$-net of the $n$-dimensional sphere.

\begin{fact}[{\label{fac:epsnet}\cite[Lemma 2.3.4]{tao2012topics}}]
For every $d \ge 1$ and any $0<\varepsilon<1/2$ there exists an $\varepsilon$-net of the sphere $S^{d-1}$ of cardinality at most
$(3/\varepsilon)^d$.
\end{fact}

The following claim gives a formula for the pdf of the product of two real-valued random variables.

\begin{claim}[{\cite[Page 134, Theorem 3]{rohatgi_saleh_2015}}]\label{clm:productpdf}
Let $\bx,\by$ be two real-valued random variables and $f$ be the pdf of $(\bx,\by)$. Then the pdf of $\bz=\bx\cdot\by$ is given by
  \[
  g\br{z}=\int_{-\infty}^{\infty}f\br{x,\frac{z}{x}}\cdot \frac{1}{|x|}dx.
  \]
\end{claim}

\begin{theorem}[{\cite[Theorem 1]{10.1214/EJP.v15-798}}]
\label{thm:tracy}
Let $\bA\sim\H_d$. For every  $0<\eta<1$, 
it holds that
\[
\Pr \Big[{\lambda_{\max}\br{\bA}\in 2\sqrt{d} \Br{1\pm\eta}}\Big]\geq 1- C\cdot e^{-d\eta^{3/2}/C},
\]
for some absolute constant $C>0$. 
\end{theorem}

\section{Proof of main theorem}

The core of the argument is in the following lemma, bounding $q(2 \sqrt{nd})$ where $q$ is the pdf of the largest eigenvalue of the matrix showing up in Eq.~\eqref{eq:main theorem spectrahedron}. We will later show that this value is essentially the same as $\GSA(\Sh)$, where $\Sh$ is the spectrahedron in the statement of the theorem. 

\begin{lemma}
\label{lem:main}
For $n,d\geq 1$ and $A^{(1)},\ldots, A^{(n)} \in \mathrm{Sym}_d$, let $q(\cdot)$ be the probability density function of 
\[
 \lambda_{\max}\br{ \sum_i\bx_i A^{(i)}} \; ,
\]
 where $\bx=(\bx_1,\ldots,\bx_n)$ is a random vector and each entry is i.i.d.\ drawn from $\G(0,1)$. 
If $\bA^{(1)},\ldots,\bA^{(n)}$ are i.i.d.~drawn from the $d\times d$ Gaussian orthogonal ensemble, then
$q(2 \sqrt{nd})\geq~c\cdot\sqrt{1/d}$ with probability at least $1-C \exp(-d n^{-3/4} / C)$ (over the choice of $\bA^{(1)},\ldots,\bA^{(n)}$) where $c, C>0$ are universal constants.  Moreover, for any integer $d$ and any   $d\times d$ matrices $A^{(1)},\ldots,A^{(n)}$, $q(2 \sqrt{nd}) \leq 1/(2\sqrt{\pi d})$.
\end{lemma}

\begin{proof}
Let $\by\sim S^{n-1}$ be chosen uniformly from the unit sphere
and for matrices $A^{(1)},\ldots, A^{(n)}$, denote by $p$ the pdf of $\lambda_{\max}\br{\sum_i\by_i A^{(i)}}$.
Let $\boldr\sim \chi_n$ and notice that $\boldr \by$ is distributed like $\bx$ (since both are spherically symmetric and by definition, have equally distributed norms). Denote by $h$ the pdf of $\boldr$.
By Claim~\ref{clm:productpdf}, we have
\begin{eqnarray}
  q\br{2 \sqrt{ nd} } =
  \int_{-\infty}^{\infty}h\br{2\sqrt{nd}/z}p\br{z}\frac{1}{|z|}dz
   \; .
  \label{eq:equivq}
\end{eqnarray}
Using Item 2 of Fact~\ref{fact:chidistribution},
$h(2\sqrt{nd}/z)/|z| \leq 1/(2\sqrt{\pi d})$ for all $z$.  
Hence Eq.~\eqref{eq:equivq} can be bounded as  
$(1/(2\sqrt{\pi d}))\cdot \int_{-\infty}^{\infty}p\br{z}dz  = 1/(2\sqrt{\pi d})$, establishing the claimed upper bound on $q$. 

To prove the lower bound on $q$, let $\bA^{(1)},\ldots,\bA^{(n)}\sim \H_d$ be $n$ matrices chosen i.i.d.\ from the Gaussian orthogonal ensemble. 
Observe that by Theorem~\ref{thm:tracy}, we have
        \begin{align}
        \label{eq:rangeoflambdamax}
        \Pr\Big[\lambda_{\max}\br{\sum_{i=1}^n \by_i \bA^{(i)}} \in
        I\Big]\geq 1- C \exp(-d n^{-3/4} / C) \; ,
        \end{align}        
        where 
        \[
        I=2\sqrt{d}\cdot[1\pm c/\sqrt{n}] \; ,
        \]
        for some universal constants $C,c>0$.
Define the set of~matrices
\begin{equation*}\label{eqn:G}
  G=\set{\br{A^{(1)},\ldots, A^{(n)}}:\Pr\Big[\lambda_{\max}\br{\sum_{i=1}^n \by_i A^{(i)}}\in I\Big]\geq \frac{1}{2}}.
\end{equation*}
Then, using the definition of $G$ and Eq.~\eqref{eq:rangeoflambdamax}, we have
\[
\Pr\Big[ \br{\bA^{(1)},\ldots,\bA^{(n)}}\in G \Big]
\geq
1-2 C \exp(-d n^{-3/4} / C) \; .
\]
Now fix any $(A^{(1)},\ldots, A^{(n)})\in G$. By definition of $G$, 
$\int_I p\br{z} dz \geq 1/2$, and therefore the right-hand side of Eq.~\eqref{eq:equivq} is at least

\begin{eqnarray}
   \int_{I}h\br{2\sqrt{nd}/z}p\br{z}\frac{1}{z}dz
   \geq
   c \cdot\int_{I}p\br{z}\frac{1}{z}dz
   \geq \frac{c}{2\sqrt{d}(1+c/\sqrt{n})}\cdot\int_Ip(z)dz\geq
  \frac{c}{5\sqrt{d}} \; ,
   \label{eqn:1}
\end{eqnarray}
for some absolute constant $c>0$, where we used Item 1 of Fact~\ref{fact:chidistribution} to conclude that $h(2\sqrt{nd}/z)\geq c$ for all $z \in I$. 
\end{proof}

We next relate $q(2 \sqrt{nd})$ to $\GSA(\Sh)$. 
For a vector $v\in S^{d-1}$, and $d \times d$ symmetric matrices $A^{(1)},\ldots,A^{(n)}$, define the vector
  \begin{equation}\label{eqn:wv}
 W_v=\big(v^T A^{(1)} v,v^T A^{(2)} v,\ldots,v^T A^{(n)} v\big) \in \R^n \; .
 \end{equation}
Notice that $\Sh$ can be written as
\[
\Sh=
\Big\{x\in \R^n:\sum_i x_i A^{(i)}\preceq 2 \sqrt{nd}\cdot\id \Big\} = 
\Big\{x\in \R^n: \forall v \in S^{d-1},~ \langle x, W_v \rangle \leq 2 \sqrt{nd} \Big\} \; .
\]
We say that $A^{(1)},\ldots,A^{(n)}$ are \emph{good} if
\[
\forall v\in S^{d-1},~\frac{1}{2}\sqrt{n} \leq \norm{W_v} \leq 2\sqrt{n}\; .
\]
\begin{lemma}
\label{lem:technicallemma}
There exists a constant $C\geq 1$ such that
for all integers $n$ and $d\leq n/C$, 
 random matrices $\bA^{(1)},\ldots,\bA^{(n)}$ drawn i.i.d.\ from $\H_d$ are good with probability at least $1-\exp\br{-n/50}$. 
  \end{lemma}

 \begin{proof}
 For a fixed $v\in S^{d-1}$, we claim that
\begin{align}
\label{eq:part1bound}
\Pr [{n\leq \norm{\bW_v}^2\leq3n}]\geq 1-2\exp\br{-n/40}. \; 
\end{align}
 To see this, observe that by definition of the Gaussian orthogonal ensemble, for $\bA\sim\H_d$ and unit vector $v\in\reals^d$, $v^T\bA v = \sum_{i,j} v_i v_j \bA_{i,j}$
 is distributed according to
 $$
 \br{4\sum_{i<j}v_i^2v_j^2+2\sum_iv_i^4}^{1/2}\cdot \G(0,1)=\sqrt{2}\cdot \G(0,1).
 $$
Therefore, each entry in $\bW_v$ is distributed according to $\G(0,2)$, and 
Lemma~\ref{lem:chi2} implies  Eq.~\eqref{eq:part1bound}. We next  prove that with high probability (over the $\bA^{(i)}$s), for \emph{every} unit vector $z$,  $\|\bW_z\|$ is large. First, by Fact~\ref{fac:epsnet}, there exists a set
$\V=\{v_1,\ldots,v_{10^{5d}}\}\subseteq \R^d$
of unit vectors that form a $10^{-4}$-net of the unit Euclidean sphere. Applying a union bound on $\V$, we have

 \begin{equation}\label{eqn:wvwv}
 \Pr [{\forall v\in\V:n\leq \norm{\bW_v}^2\leq 3n}]\geq 1-2\exp\br{-n/40}\cdot 10^{5d}\geq 1-\exp\br{-n/50} \; ,
 \end{equation}
here we used that $d\leq n/C$ for a sufficiently large $C$.

To conclude the proof, it suffices to show that if $\bA^{(1)},\ldots,\bA^{(n)}$ are such that
\[
\forall v\in\V, ~n \leq \norm{\bW_v}^2 \leq 3n \; ,
\]
then also
\[
\forall z\in S^{d-1},~\norm{ \bW_z } \geq\frac{1}{2}\sqrt{n} \; .
\]
Let $\bn_{\max}=\max_{z\in S^{d-1}}\norm{\bW_z}$ and $\bn_{\min}=\min_{z\in S^{d-1}}\norm{\bW_z}$. Let $\bz_{\max}$ and $\bz_{\min}$ be the vectors achieving the maximum and the minimum respectively. Let $\bv_{\max}$ and $\bv_{\min}$ be the vectors in $\V$ that are closest to $\bz_{\max}$ and $\bz_{\min}$, respectively.
For any vectors $z,v\in S^{d-1}$ with $\norm{z-v}\leq 10^{-4}$,  applying the spectral decomposition of $zz^T-vv^T$, there exist unit vectors $u_1,u_2$ and $0\leq\lambda\leq\frac{1}{100}$ such that
\begin{equation}\label{eqn:vz}
  zz^T-vv^T=\lambda\cdot \br{u_1u^T_1-u_2u_2^T} \; .
\end{equation}
Hence
\begin{align*}
 \norm{\bW_{z}-\bW_{v}}^2 =
  \sum_{i=1}^n\br{z^T\bA^{(i)}z-v^T \bA^{(i)}v}^2 &=\sum_{i=1}^n\br{\Tr \br{\bA^{(i)}\br{zz^T-vv^T}}}^2 \\
  &\leq\frac{1}{10^4}\sum_{i=1}^n\br{u_1^T \bA^{(i)}u_1-u_2^T\bA^{(i)}u_2}^2\\
  &\leq \frac{1}{5000}\sum_{i=1}^n\br{\br{u_1^T \bA^{(i)}u_1}^2+\br{u_2^T\bA^{(i)}u_2}^2}\\
  &\leq\frac{\bn_{\max}^2}{2500} \; .
\end{align*}
Choosing $z=\bz_{\max}$ and $v=\bv_{\max}$, we have
\[
\norm{\bW_{\bz_{\max}}}\leq\norm{\bW_{\bv_{\max}}}+\frac{\bn_{\max}}{50} \; .
\]
Now, since $\norm{\bW_{\bz_{\max}}}=\bb_{\max}$, we have
\[
\bn_{\max}
\leq
\frac{50}{49}\norm{\bW_{\bv_{\max}}}
\leq
\frac{50}{49} \sqrt{3n} \le 2 \sqrt{n}
\; .
\]
Similarly, we set $z=\bz_{\min}$ and $v=\bv_{\min}$ and obtain
\[
\bn_{\min}\geq\norm{\bW_{\bv_{\min}}}-\frac{\bn_{\max}}{50}
\geq
\sqrt{n} - \frac{1}{25} \sqrt{n} > \frac{1}{2} \sqrt{n} \; .
\]
This concludes the result.
\end{proof}

For the following claim, we define the inner and outer shells of $\Sh$ as
\begin{align*}
\mathcal{D}^{\mathrm{in}}_{\delta} 
&=
\Big\{x:\lambda_{\max}\Big(\sum_ix_iA^{(i)}\Big)\in\sqrt{n}\cdot \Br{2\sqrt{d}-\delta,2\sqrt{d}} \Big\} \; , \\
\mathcal{D}^{\mathrm{out}}_{\delta} 
&=
\Big\{x:\lambda_{\max}\Big(\sum_ix_iA^{(i)}\Big)\in\sqrt{n}\cdot \Br{2\sqrt{d},2\sqrt{d}+\delta} \Big\} \; .
\end{align*}
Also recall the inner and outer neighborhoods of $\Sh$, defined as
\begin{align*}
\Sh_{\delta}^{\mathrm{in}}
&=\set{x\in\Sh:\exists y\notin\Sh:\norm{x-y} \leq \delta} \; , \\
\Sh_{\delta}^{\mathrm{out}}
&=\set{x\notin\Sh:\exists y\in\Sh:\norm{x-y}\leq \delta} \; .
\end{align*}

\begin{claim}
\label{claim:DdeltaTdelta}
For sufficiently small $\delta>0$ and any good $A^{(1)},\ldots,A^{(n)}$, we have $\mathcal{D}^{\mathrm{in}}_{\delta}\subseteq\Sh_{4\delta}^{\mathrm{in}}$ and
$\Sh_{\delta}^{\mathrm{out}} \subseteq \mathcal{D}^{\mathrm{out}}_{2\delta}$.
\end{claim}
\begin{proof}
 For every $x\in \mathcal{D}^{\mathrm{in}}_{\delta}$, let $v$ be a unit eigenvector of $\sum_i x_iA^{(i)}$ with the eigenvalue $\lambda_{\max}(\sum_i x_iA^{\br{i}})$. 
Therefore, 
\[
\langle x, W_v \rangle =
v^T \Big(\sum x_i A^{(i)}\Big) v \geq(2\sqrt{d}-\delta)\sqrt{n} \; .
\]          
Setting $y=2\delta\sqrt{n}W_v / \norm{W_v}^2$, we have
  \begin{align*}
\langle x+y, W_v  \rangle     
     &= \langle x,W_v \rangle + 2\delta\sqrt{n} \geq
      \br{2\sqrt{d}-\delta}\sqrt{n}+2\delta\sqrt{n} =\br{2\sqrt{d}+\delta}\sqrt{n}
      \; ,
  \end{align*}
 and so $x+y\notin\Sh$. 
 Moreover, since $A^{(1)},\ldots,A^{(n)}$ are good, 
 $
 \|y\| = 2\delta\sqrt{n}/\norm{W_v} \leq 4\delta
 $
 and therefore $x \in \Sh_{4\delta}^{\mathrm{in}}$, as desired. 
 For the other containment, let $x\in \Sh_{\delta}^{\mathrm{out}}$.
 Then for any unit vector $v$, by Cauchy-Schwarz and using $\|W_v\| \le 2 \sqrt{n}$,
\[
\langle x, W_v  \rangle \le 2 \sqrt{nd} + 2 \delta \sqrt{n} \; ,
\]
implying that $x \in \mathcal{D}^{\mathrm{out}}_{2\delta}$, as desired. 
 \end{proof}

We now prove our main theorem. 

\begin{proof}[Proof of Theorem~\ref{thm:gsa}]
By Lemmas~\ref{lem:main} and~\ref{lem:technicallemma}, 
if $\bA^{(1)},\ldots,\bA^{(n)}$ are i.i.d.~drawn from the $d\times d$ Gaussian orthogonal ensemble, then
with probability at least $1-C \exp(-d n^{-3/4} / C)$,
we have that $q(2 \sqrt{nd})\geq~c\cdot\sqrt{1/d}$ (where $q(\cdot)$ is as defined in Lemma~\ref{lem:main}) and
that $\bA^{(1)},\ldots,\bA^{(n)}$ are good, 
where $c,C>0$ are some constants. 
Since $q(\cdot)$ is continuous, the former 
implies that $\G^n(\De^{\mathrm{in}}_{\delta})\geq c\delta \sqrt{n/(2d)}$ 
for sufficiently small $\delta>0$. Thus,
$\G^n(\Sh_{4\delta}^{\mathrm{in}}) \ge c\delta \sqrt{n/(2d)}$  
by Claim~\ref{claim:DdeltaTdelta}. By definition of  $\GSA(S)=\lim_{\delta\rightarrow 0}\G^n(S^{\mathrm{in}}_{\delta})/\delta$, we obtain the desired lower bound on $\GSA(\Sh)$. 
Similarly, by Lemmas~\ref{lem:main} and~\ref{lem:technicallemma}, 
if $\bA^{(1)},\ldots,\bA^{(n)}$ are i.i.d.~drawn from the $d\times d$ Gaussian orthogonal ensemble, then
with probability at least $1-\exp\br{-n/50}$,
$\G^n(\De^{\mathrm{out}}_{\delta}) \le \delta \sqrt{n}/(\sqrt{\pi d})$
for sufficiently small $\delta>0$. Thus,
$\G^n(\Sh^\text{out}_{\delta/2}) \le \delta \sqrt{n}/(\sqrt{\pi d})$
by Claim~\ref{claim:DdeltaTdelta}. We complete the proof using  $\GSA(S)=\lim_{\delta\rightarrow 0}{\G^n(S^{\text{out}}_{\delta}})/\delta$.
\end{proof}

 \paragraph{Acknowledgements.} We thank Daniel Kane, Assaf Naor, Fedor Nazarov, and Yiming Zhao for useful correspondence. O.R. is supported by the Simons Collaboration on Algorithms and Geometry, a Simons Investigator Award, and by the National Science Foundation (NSF) under Grant No.~CCF-1814524. P.Y. is supported by the National Key R\&D Program of China 2018YFB1003202, National Natural Science Foundation of China (Grant No. 61972191), the Program for Innovative Talents and Entrepreneur in Jiangsu and Anhui Initiative
in Quantum Information Technologies Grant No.~AHY150100.

\end{document}